\documentclass[onefignum,onetabnum]{siamart190516}
\usepackage{amsmath,amssymb,amsxtra}
\usepackage{mathrsfs}
\usepackage{float,verbatim}
\usepackage{threeparttable,booktabs}
\usepackage{graphicx}
\usepackage{epstopdf}
\usepackage{mathtools}
\usepackage{subfig}
\usepackage{algorithm,algorithmic}
\usepackage{color}
\usepackage{latexsym,float,tikz}
\graphicspath{{./pics/}}

\usepackage{array}
\newtheorem{remark}{Remark}[section]
\newtheorem{assumption}{Assumption}[section]

\usepackage{graphicx}
\usepackage{dsfont}
\usepackage{epstopdf}
\graphicspath{{./pics/}}
\usepackage{booktabs}
\usepackage{threeparttable}
\usepackage{verbatim,bm}
\newcommand{\op}{\operatorname}
\newcommand{\dx}{\mathrm{d}x}
\newcommand{\norm}[1]{\left\|#1\right\|}
\newcommand{\normE}[2]{\norm{#1}_{H^{1}_{\kappa}\prnt{#2}}}
\newcommand{\normL}[2]{\norm{#1}_{L^2\prnt{#2}}}
\newcommand{\normb}[1]{\norm{#1}_{a_H,D}}
\newcommand{\dy}{\mathrm{d}y}
\newcommand{\prnt}[1]{\left( #1 \right)}

\numberwithin{equation}{section}

\allowdisplaybreaks

\allowdisplaybreaks

\title{On the convergence analysis of MsFEM with oversampling: Interpolation error}

\author{Guanglian Li\thanks{Department of Mathematics, The University of Hong Kong, Pok Fu Lam Road, Hong Kong SAR. ({\tt{lotusli@maths.hku.hk}})
This research is supported by the Hong Kong Research Grants Council, Early Career Scheme (project 27301921).}}

\date{\today}

\begin{document}

\maketitle

\begin{abstract}
In this paper, we investigate the approximation properties of two types of multiscale finite element methods with oversampling as proposed in [Hou \& Wu, {\textit{J. Comput. Phys.}}, 1997] and [Efendiev, Hou \& Wu, \textit{SIAM J. Numer. Anal.}, 2000] without scale separation. 
We develop a general interpolation error analysis for elliptic problems with highly oscillatory rough coefficients, under the assumption of the existence of a macroscopic problem with suitable $L^2$-accuracy. The distinct features of the analysis, in the setting of highly oscillatory periodic coefficients, include: (i) The analysis is independent of the first-order corrector or the solutions to the cell problems, and thus independent of their regularity properties; (ii) The analysis only involves the homogenized solution and its minimal regularity. We derive an interpolation error $\mathcal{O}\left(H+\frac{\epsilon}{H}\right)$ with $\epsilon$ and $H$ being the period size and the coarse mesh size, respectively, when the oversampling domain includes one layer of elements from the target coarse element.
\end{abstract}


\section{Introduction}
Let $D\subset
\mathbb{R}^d$ ($d=1,2,3$) be an open bounded convex domain with a boundary $\Gamma =\partial D$. In this work we consider the following elliptic boundary value problem with highly oscillatory rough coefficients:
 \begin{equation*}
\left\{\begin{aligned}
-\nabla\cdot\left(\kappa^{\epsilon}\nabla u^{\epsilon}\right)&=f &&\quad\text{ in }D,\\
u^{\epsilon}&=0 &&\quad\text{ on } \Gamma,
\end{aligned}\right.
\end{equation*}
with $\kappa^{\epsilon}\in (L^{\infty}(D))^{d\times d}$ being a uniformly elliptic, bounded and matrix-valued permeability coefficient. The parameter $\epsilon\ll 1$ denotes the microscopic scale underlying the problem. Partial differential equations (PDEs) with multiple scales arise in many practical applications, e.g., photonic crystals, metamaterials, waveguides, and transmission lines. The presence of a small parameter $\epsilon$ makes standard numerical methods with (local) polynomial basis functions computationally infeasible or prohibitive. These observations have motivated the intensive development of homogenization theory and multiscale numerical methods in the past few decades; see, e.g., \cite{MR2916381,MR4298217,MR4891112,MR4754301,MR4628018,MR2477579,MR4520724,MR4700411,fu2019edge,MR4490299,MR4372648,MR4191211,MR3971243,MR3177856}.

Among the existing numerical methods  for solving problems with highly oscillatory coefficients, multiscale finite element methods (MsFEMs) proposed by Hou and Wu \cite{MR1455261} represent one of the most popular and successful class of multiscale numerical methods. This class of methods enjoys ease of implementation and demonstrates impressive empirical performance across a wide range of applied problems, and thus has received a lot of attention. However, the theoretical analysis of the methods has lagged behind. The method was analyzed in \cite{MR1642758} and \cite{MR1740386,MR2119937} for MsFEMs without oversampling and with oversampling for elliptic problems with highly-oscillatory periodic coefficients under certain assumptions, respectively. Since then, further convergence analysis under more relaxed assumptions has also been proposed (see, e.g., \cite{MR2399546,MR2982460,MR3123820,MR4686842}).

In this work, we shall develop a general interpolation error analysis for MsFEMs with oversampling under the assumption of the existence of a macroscopic equation with $L^2$-accuracy, cf. Assumption \ref{ass:macroscale}. Together with the stability of the discrete scheme, the interpolation error analysis would give a full analysis of MsFEMs with oversampling. It is worth noting that in the setting of highly oscillatory periodic coefficients, the analysis does not rely on the solutions to the cell problems or the first-order corrector, which have played a fundamental role in the existing theoretical studies. Specifically, the standard approach for the case of highly oscillatory periodic coefficients proceeds in three steps \cite{MR1642758}. First, one utilizes multiscale basis functions to interpolate the homogenized solution as an approximator to the exact solution $u^{\epsilon}$ directly. Second, one performs a two-scale expansion of  the exact solution in the global domain and the approximator in each subdomain arising from the two-scale expansion for multiscale basis functions. Third and finally, one expresses the error in terms of the two-scale expansion and then estimates each term separately. Hence the standard analysis strategy relies heavily on the regularity of the homogenized solution, the solutions to the cell problem, and the first-order corrector.

In contrast to the standard paradigm, we shall develop a new proof that relies only on the local approximator \eqref{eq:loc-def} in each oversampled domain and thus the local error estimate therein, cf. Lemma \ref{lem:local-approximate}. To derive the local error estimate, we first derive an $L^2$- a priori estimate of the harmonic extension for the macroscopic equation, given the $L^2$-Dirichlet boundary data, cf. Theorem \ref{thm:veryweak-homogenized}, and then employ several auxiliary functions defined by the local elliptic operator and local macroscopic operator subject to certain linear boundary conditions to obtain an $L^2$-error estimate on the oversampled domain. Then we derive the $H^1$-error estimate over each coarse element using a Caccioppoli type estimate. The derivation of the local error estimate is inspired by the convergence analysis in \cite[Lemma 4.4]{MR3939320}, in which partition of unity functions are employed in the construction of multiscale basis functions to enforce conformity and play a critical role in the analysis. In sharp contrast, the multiscale space with oversampling is nonconforming and the local error estimate developed in Lemma \ref{lem:local-approximate} can avoid the usage of partition of unity functions. Next, we employ an auxiliary global function defined elementwisely using the local approximator, which enjoys a good approximation property to the exact solution. However, it does not belong to the multiscale FEM space since it may take multiple values over each node. To obtain a good approximator in the multiscale FEM space, we utilize an alternative global function that uses the average of multiple values on each node. We establish in Theorems \ref{theorem:global-approximation} and \ref{theorem:global-approximation-type2} an error estimate of MsFEMs with an arbitrary number of fine-scale oversampled layers for MsFEM with oversampling proposed in \cite{MR1455261} and \cite{MR1740386}, respectively. This is the main theoretical achievement of the work. Furthermore, in the setting with highly oscillatory (locally) periodic coefficients, we show that Assumption \ref{ass:macroscale} is indeed valid, and obtain a global error estimate $\mathcal{O}(H+\frac{\epsilon}{H})$ when the oversampled layers contain one coarse element, where $H$ is the coarse mesh size. This result is consistent with the results from the work \cite{MR1740386}, which is derived using the standard approach in the context of highly oscillatory periodic coefficients.

The remainder of this paper is organized as follows. We describe in Section \ref{sec:problem} the model setting of elliptic problems with heterogeneous coefficients, and its macroscopic model, and then MsFEMs with oversampling in Section \ref{sec:method}. Next, in Section \ref{sec:estimate} we present the main interpolation error analysis. Moreover, we discuss in Section \ref{sec:structure} the convergence for high oscillatory periodic and locally periodic coefficients. Finally, we conclude with a brief summary in Section \ref{sec:conc}. Throughout, the notation $A\lesssim B$ denotes $A\leq C B$ for some constant
$C$ independent of the microscale $\epsilon$, fine mesh size $h$, and coarse mesh size $H$.


\section{Problem setting}\label{sec:problem}
In this section, we formulate the heterogeneous elliptic problem. We look for
a function $u^{\epsilon}\in V:=H^{1}_{0}(D)$ such that
\begin{equation}\label{eq:model}
\left\{\begin{aligned}
\mathcal{L}u^{\epsilon}:=-\nabla\cdot\left(\kappa^{\epsilon}\nabla u^{\epsilon}\right)&=f &&\quad\text{ in }D,\\
u^{\epsilon}&=0 &&\quad\text{ on } \Gamma,
\end{aligned}\right.
\end{equation}
where the force term $f\in L^2(D)$ and the permeability coefficient
$\kappa^{\epsilon}\in \mathcal{M}(\alpha,\beta;D)$, with
\begin{align*}
\mathcal{M}(\alpha,\beta; D):=\left\{A\in (L^{\infty}(D))^{d\times d}: \alpha|\bm{\xi}|^2\leq (A(x)\bm{\xi},\bm{\xi})\leq\beta|\bm{\xi}|^2,\;\forall\bm{\xi}\in\mathbb{R}^d\text{ and }\op{a.e.}, x\in D\right\},
\end{align*}
 for some $\beta\geq\alpha>0$. Here, the notation $(\cdot,\cdot)$ denotes the inner product in $\mathbb{R}^d$ and $|\cdot|$ the Euclidean norm. Note that the presence of the small parameter $\epsilon$ in the permeability coefficient $\kappa^\epsilon$ renders directly
 solving problem \eqref{eq:model} very challenging, since resolving the problem to the finest scale would incur
prohibitively high computational cost.

Then the variational formulation of problem \eqref{eq:model} reads: find $u^{\epsilon}\in V$ such that
	\begin{equation}\label{eq:weak}
	a(u^{\epsilon},v)=b(v),\quad \forall v\in V,
	\end{equation}
    with the bilinear form $a(\cdot,\cdot):V\times V\to\mathbb{R}$ and linear form $b(\cdot):V\to\mathbb{R}$, defined respectively, by
\begin{align*}
a(w,v):=\int_{D}\kappa^{\epsilon}\nabla w\cdot \nabla v\dx\quad \text{and}\quad
b(v):=\int_{D}fv \dx,\quad \forall w,v\in V.
\end{align*}
 The Lax-Milgram theorem implies the well-posedness of problem \eqref{eq:weak}.
 Let $\normE{\cdot}{\Omega}$ be the associate norm for any subdomain $\Omega\subset D$, i.e.,
 \begin{align*}
\normE{v}{\Omega}:=\left(\int_{\Omega}\kappa^{\epsilon}|\nabla v|^2\mathrm{d}x\right)^{1/2}.
 \end{align*}

Next, we give the assumption on the existence of a macroscopic problem associated with problem \eqref{eq:model} which is the only assumption we make on the problem setting.
\begin{assumption}[Macroscopic problem]\label{ass:macroscale}
There exists a macroscopic problem associated with problem \eqref{eq:model}, with the corresponding solution $u_0\in V\cap H^2(D)$ satisfying
\begin{equation}\label{eq:model-hom-abstract}
\left\{\begin{aligned}
\overline{\mathcal{L}}u_0:=-\nabla\cdot\left(\overline{\kappa}\nabla u_0\right)&=f &&\quad\text{ in }D,\\
u_0&=0 &&\quad\text{ on } \Gamma,
\end{aligned}\right.
\end{equation}
for some $\overline{\kappa}\in \mathcal{M}(\overline{\alpha},\overline{\beta};D)\cap (W^{1,\infty}(D))^{d\times d}$ with parameters $0<\overline{\alpha}\leq \overline{\beta}$. Moreover, there exists $\delta>0$ such that
\begin{align}\label{eq:1/3-abstract}
\normL{u^{\epsilon}-u_0}{D}\lesssim \epsilon^{\delta} \|f\|_{L^2(D)},
\end{align}
where the involved constant is independent of the boundary data and domain.
\end{assumption}

Note that Assumption \ref{ass:macroscale} is reasonable since there exists $\overline{\kappa}\in \mathcal{M}(\overline{\alpha},\overline{\beta};D)$ such that $\kappa^{\epsilon}$ converges to $\overline{\kappa}$ as $\epsilon\to 0$ in the sense of $H$-convergence \cite{MR1493039}, or $G$-convergence when $\kappa^{\epsilon}$ is symmetric \cite{MR348255}. This implies the strong convergence of $u^{\epsilon}\to u_0$ in $L^2(D)$-norm and weak convergence in $H^1_0(D)$-norm (up to a subsequence). Assumption \ref{ass:macroscale} can be viewed as a quantitative version of the convergence, and it is not restricted to the concept of $H$-convergence or $G$-convergence. Moreover, the uniqueness of the macroscopic problem is not needed. Assumption \ref{ass:macroscale} is the basic assumption for the finite element heterogeneous multiscale method (FE-HMM) (see, e.g., \cite[Section 4]{MR2916381}), with convergence to $u_0$ as its fundamental objective. 
Note also that the macroscopic elliptic operator $\overline{\mathcal{L}}$ does not necessarily take the same form as the microscale one $\mathcal{L}$. The analysis below allows more general elliptic operators, e.g., the homogenized problem in \cite{MR1493040}.

We use also the concept of macroscopic component, which plays a crucial role in the subsequent analysis.
\begin{definition}[Macroscopic component]\label{def:macro-comp}
For any convex subdomain $\Omega\subset D$, let $v^{\epsilon}\in H^1(\Omega)$ and $g\in L^2(\Omega)$ be such that 
\begin{align*}
 \mathcal{L}v^{\epsilon}=g\quad \text{in }\Omega.
\end{align*}
Then $v_0\in H^1(\Omega)$ is called the macroscopic component of $v^{\epsilon}$ on $\Omega$, if it satisfies
\begin{equation*}
\left\{
\begin{aligned}
 \overline{\mathcal{L}}v_0&=g&&\quad\text{ in }\Omega,\\
v_0&=v^{\epsilon} &&\quad\text{ on } \partial\Omega.
\end{aligned}
\right.
\end{equation*}
\end{definition}
\section{MsFEM with oversampling}\label{sec:method}
In this section, we recall two popular types of MsFEMs with oversampling \cite{MR1740386,MR1455261}. First, we describe the finite element (FE) discretization of problem \eqref{eq:weak}.
Let $\mathcal{T}_H$ be a quasi-uniform partition of the domain $D$ into a Cartesian mesh composed of closed quadrilaterals ($d=2$) or closed hexahedra ($d=3$) without hanging nodes and with a coarse mesh size $H$ such that 
$\mathcal{T}_h$ is its subdivision by regular, possibly nonuniform partition into closed  quadrilaterals (d=2) or closed hexahedra (d=3) (with a fine mesh size $h$). Let $V_H$ be the standard $H^1$-conforming piecewise $d$-linear polynomial space
 \begin{align*}
V_H=\{v\in V:v|_K\in Q^1(K) \quad\forall K\in\mathcal{T}_H\}.
 \end{align*}
Let $\mathcal{N}_H$ and $\mathcal{E}_H$ be the sets of interior nodes and interior edges for the coarse mesh $\mathcal{T}_H$, respectively. We define the index set $\mathcal{J}_H$ such that $\mathcal{N}_H:=\{x_p:p\in\mathcal{J}_H\}$. Let $K^{m}$ be the subdomain associated with an element $K\in \mathcal{T}_H$, defined by
$K^m=\bigcup\left\{\tau\in\mathcal{T}_h:\text{dist}(\tau,K)\leq mh\right\}$. We denote by $d_K$ the number of vertices on $K^m$. Let $\{x_{K,j}\}_{j=1}^{d_K}$ and $\{x_{K^m,j}\}_{j=1}^{d_K}$ be the nodal points on each coarse element $K$ and its oversampled subdomain $K^m$, respectively. See  Fig. \ref{fig:grid} for a schematic illustration of the two-scale mesh in the two-dimensional case. We denote by $L_{K^m}:=\op{diam}(K^m)$ the diameter of $K^m$. By the construction, we have
\begin{align}\label{eq:diam-subdomain}
L_{K^m}\leq \sqrt{d}(H+2mh).
\end{align}
Throughout, we assume
\begin{align}\label{eq:c-diameter}
L_{K^m}\approx H, \quad\forall K\in\mathcal{T}_H,\quad \text{and}\quad \epsilon\ll H.
\end{align}
Moreover, we assume the following finite-overlapping property: there exists a finite number $\Lambda>1$ independent of $(h,H)$ such that
\begin{align}\label{eq:finite-overlap}
\Lambda=\max\left\{\#\Lambda_{\tau}:\tau\in\mathcal{T}_h\right\}, \quad \text{ with }\Lambda_{\tau}=\{K'\in\mathcal{T}_H:{\tau}\cap{K'}^m\neq \emptyset\}.
\end{align}
Next, we define a sequence of cutoff functions $\{\eta_K\}_{K\in\mathcal{T}_H}$ subordinate to the cover $\{K^m\}_{K\in\mathcal{T}_H}$, which satisfies that for some $C_{\infty}$ and $C_{\op{G}}$ independent of $(\epsilon,h,H)$, the following properties hold
\begin{equation}\label{eq:pum-property}
\begin{aligned}
\eta_{K}&=1 \text{ in } {K},\quad {\text{supp}(\eta_{K})}\subset {K^{m}},\quad
 \| \eta_{K}\|_{L^{\infty}({K^{m}})}\leq C_{\infty},\quad
\|\nabla \eta_{K}\|_{L^{\infty}({K^m})}\leq (mh)^{-1}C_{\op{G}}.
\end{aligned}
\end{equation}

Then we define a nonconforming multiscale FE space. To eliminate the so-called resonance errors \cite{MR1898136} (i.e., the numerical accuracy deteriorates as $H$ approaches $\epsilon$), an oversampling strategy can be adopted in the construction.
Let $\gamma_{K}^{m}:=\partial K^m\backslash\Gamma$ and $\Gamma_K^m:=\partial K^m\cap\Gamma$ denote the interior boundary and exterior boundary for each subdomain $K^m$. We first construct intermediate bases $\{\Phi_{K,i}^{\operatorname{ms},m}\}_{i=1}^{d_K}$ in each oversampled domain $K^m$, by finding $\Phi_{K,i}^{\operatorname{ms},m}\in H^1_{\Gamma,0}(K^m)$ such that
 \begin{equation}\label{eq:lp-os}
 \left\{
 \begin{aligned}
\mathcal{L}\Phi_{K,i}^{\operatorname{ms},m}&=0 &&\text{in } K^m, \\
\Phi_{K,i}^{\operatorname{ms},m}&=\psi_{K^m,i} &&\text{on }\gamma_{K}^{m},
\end{aligned}\right.
\end{equation}
where $\psi_{K^m,i}$ denotes all $d$-linear functions on the subdomain $K^m$ satisfying $\psi_{K^m,i}(x_{K^m,j})=\delta_{ij}$, and $H^1_{\Gamma,0}(K^m):=\{v\in H^1(K^m):v|_{\Gamma_{K}^{m}}=0\}$.

Now we can construct the first type of local multiscale (LMS) basis functions (Type-1).
\begin{definition}[LMS-type 1]
$\{\phi_{K,p}^{\operatorname{ms},m}\}_{p=1}^{d_K}$
is derived from $\{\Phi_{K,i}^{\operatorname{ms},m}\}_{i=1}^{d_K}$ to ensure the Lagrange interpolation property over $
 \{x_{K,p}\}_{p=1}^{d_K}$,
 i.e.,
\begin{equation}\label{eq:lagrange-property}	\phi_{K,p}^{\operatorname{ms},m}=\sum_{j=1}^{d_K}c_{pj}\Phi_{K,j}^{\operatorname{ms},m}|_K,
\end{equation}
where $c_{pj}\in\mathbb{R}$ are constants determined by the condition $\phi_{K,p}^{\operatorname{ms},m}({x}_{K,q})=\delta_{pq}$ for all $1\leq q,p\leq d_K$. 
\end{definition}

Note that the H\"{o}lder estimate for elliptic problems \cite[Theorems 8.22 and 8.29]{MR1814364} ensures that $\Phi_{K,j}^{\operatorname{ms},m}\in C^{0,\alpha}(K^m)$ for some $\alpha\in (0,1)$. Hence, \eqref{eq:lagrange-property} is well defined.

Note that the multiscale basis functions proposed in \cite{MR1740386,MR2119937} impose Lagrange properties for the zeroth order term in \eqref{eq:lagrange-property} in the setting of highly oscillatory periodic coefficients. These two proposals are slightly different and allow for discontinuity not only along the interior edges $ \mathcal{E}_H$ but also the coarse nodes $\mathcal{N}_H$. In the same manner, we can define such multiscale basis functions using the local macroscopic problem. Let $\overline{\Phi}_{K,i}^{\operatorname{ms},m}\in H^1_{\Gamma,0}(K^m)$ be the macroscopic component of $\Phi_{K,i}^{\operatorname{ms},m}$ on the oversampled subdomain $K^m$ (i.e., the solution to \eqref{eq:lp-os} with the heterogeneous elliptic operator $\mathcal{L}$ replaced by the macroscopic operator $\overline{\mathcal{L}}$). Then we construct the second type of local multiscale basis functions.
 \begin{definition}[LMS-type 2]
 $\{{\varphi}_{K,p}^{\operatorname{ms},m}\}_{p=1}^{d_K}$ is derived
 from $\{\Phi_{K,i}^{\operatorname{ms},m}\}_{i=1}^{d_K}$ to ensure the Lagrange interpolation property of its macroscopic component over $\{x_{K,p}\}_{p=1}^{d_K}$,
 i.e.,
\begin{equation}\label{eq:lagrange-property-2}	\varphi_{K,p}^{\operatorname{ms},m}=\sum_{j=1}^{d_K}c_{pj}\Phi_{K,j}^{\operatorname{ms},m}|_K,
\end{equation}
where $c_{pj}\in\mathbb{R}$ are constants determined by the condition $$\overline{\varphi}_{K,p}^{\operatorname{ms},m}({x}_{K,q})=\sum_{j=1}^{d_K}c_{pj}\overline{\Phi}_{K,j}^{\operatorname{ms},m}({x}_{K,q})
=\delta_{pq},\quad  \forall 1\leq q,p\leq d_K,$$ 
where $\overline{\varphi}_{K,p}^{\operatorname{ms},m}$ denotes the macroscopic component of ${\varphi}_{K,p}^{\operatorname{ms},m}$ on the oversampled subdomain $K^m$.
\end{definition} 

Let $v=\phi$ or $\varphi$. The multiscale basis function associated with each interior node $x_i$ for $i\in\mathcal{J}_H$ is given by
 \begin{align*}
v^{\operatorname{ms},m}_{i}|_K:=v^{\operatorname{ms},m}_{K,j},\quad \text{for }K\in\mathcal{T}_H \text{ with } x_{K,j}=x_i\in {K}\cap\mathcal{N}_H.
 \end{align*}
 By the construction, the basis function $v^{\operatorname{ms},m}_{i}$ allows for discontinuity along the interior edges $\mathcal{E}_H$ and thus $v^{\operatorname{ms},m}_{i}\notin H^1_0(D)$. 
 
 Then the multiscale FE space with oversampling is defined by
 \begin{equation}\label{eq:v-ov}
 \begin{aligned}
V_H^{\operatorname{ms},m;v}:=\operatorname{span}\{v^{\operatorname{ms},m}_{i}: i\in\mathcal{J}_H\}
\not\subset H_0^1(D),\;v=\phi, \varphi.
\end{aligned}
\end{equation}
\begin{remark}[Domain of multiscale basis functions and linearity over $\partial K^m$]
Note that each local multiscale basis function $v^{\operatorname{ms},m}_{K,i}$ can be naturally extended to 
$K^m$ by its construction, and that its domain can be either $K$ or $K^m$ from context to context. Note also that $v^{\operatorname{ms},m}_{i}|_{\partial K^m}$ is $(d-1)$-linear for $v=\phi$ or $\varphi$.
\end{remark}

Let $v=\phi$ or $\varphi$. Since $v_{K,p}^{\operatorname{ms},m}$ is the harmonic extension over each coarse element $K$ and $d-1$-linear on its boundary $\partial K^m$, we derive
 \begin{align}\label{eq:loc-harmonic-energy}
\normE{v_{K,p}^{\operatorname{ms},m}}{K}\lesssim  (H+2mh)^{d/2-1},\quad\text{for }v=\phi,\varphi.
 \end{align} 
Next, we define a broken bilinear form $a_H(\cdot,\cdot)$ on $\oplus_{K\in\mathcal{T}_H}H^1(K)$ by
\begin{align*}
a_H(v,w):=\sum_{K\in\mathcal{T}_H}\int_{K} \kappa^{\epsilon}\nabla v\cdot\nabla w, \quad
\forall v,w\in\oplus_{K\in\mathcal{T}_H}H^1(K)
 \end{align*}
 and denote the associate norm by
$\normb{v}:=a_H(v,v)^{1/2}$. Obviously, the following equivalence of $\normb{\cdot}$ and $\normE{\cdot}{\Omega}$ in $V$ holds
\begin{align*}
\normE{v}{D}\lesssim \normb{v}\lesssim \normE{v}{D}, \quad \forall v\in V.
\end{align*}
With the multiscale space $V_H^{\operatorname{ms},m;v}$ for $v=\phi,\varphi$ and the polynomial space $V_H$, the Petrov-Galerkin multiscale FEM reads \cite{MR2119937}: find $u_H^{\operatorname{ms},m;v}\in V_H^{\operatorname{ms},m;v}$ such that
	\begin{equation}\label{eq:pet}
		a_H( u_H^{\operatorname{ms},m;v},w)=f(w),\quad \forall w\in V_H.
	\end{equation}
	\begin{figure}[!ht]
		\centering
		\begin{tikzpicture}[scale=1.5]
			\draw[step=0.25, gray, thin] (-0.4, -0.4) grid (4.4, 4.4);
			\draw[step=1.0, black, very thick] (-0.4, -0.4) grid (4.4, 4.4);
			\foreach \x in {0,...,4}
			\foreach \y in {0,...,4}{
				\fill (1.0 * \x, 1.0 * \y) circle (1.5pt);
			}
			\fill[brown, opacity=0.4] (1.0, 1.0) rectangle (2.0, 2.0);
			\node at (1.5, 1.5) {$K$};
            \node at (1, 0.8) {$x_{K,1}$};
            \node at (2, 0.8) {$x_{K,2}$};
            \node at (2, 2.2) {$x_{K,3}$};
            \node at (1, 2.2) {$x_{K,4}$};
			\draw [dashed, very thick, fill=gray, opacity=0.6] (0.5, 0.5) rectangle (2.5, 2.5);
			\node at (2.2,1.4) {$K^2$};
            \node at (0.5, 0.3) {$x_{K^2,1}$};
            \node at (2.5, 0.3) {$x_{K^2,2}$};
            \node at (2.5, 2.7) {$x_{K^2,3}$};
            \node at (0.5, 2.7) {$x_{K^2,4}$};
            \fill (0.5, 0.5) circle (1pt);
            \fill (2.5, 0.5) circle (1pt);
            \fill (2.5, 2.5) circle (1pt);
            \fill (0.5, 2.5) circle (1pt);
			\draw [dashed, very thick, fill=cyan, opacity=0.5] (3.25, 1.25) rectangle (3.5, 1.5);
			\node at (3.375, 1.375) {$\tau$};
		\end{tikzpicture}
	
		\caption{A schematic illustration of the two-scale mesh, a fine element $\tau$, a coarse element $K$, oversampling coarse element $K^m$, and $d_K$ vertices of $K$ and $K^m$ with $m=2$ and $d_K=4$.}
			\label{fig:grid}
	\end{figure}
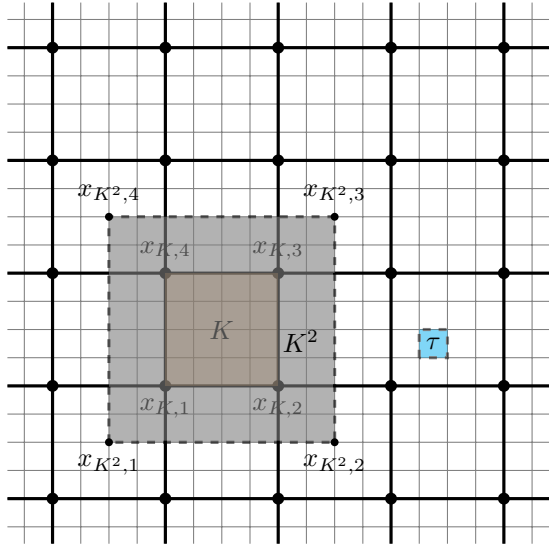

 \section{Interpolation error estimate}\label{sec:estimate}

In this section, we analyze the approximation properties of the multiscale space $V_H^{\operatorname{ms},m;v}$ defined in \eqref{eq:v-ov} to the exact solution $u^{\epsilon}$ of problem \eqref{eq:model} for $v=\phi$ or $\varphi$. The interpolation error analysis plays a fundamental role in the analysis of MsFEM. Our analysis proceeds in three steps. First, we derive $L^2$-a priori estimates for the homogeneous macroscopic operator in each oversampled subdomain $K^m$, under the $L^2$-Dirichlet data in Assumption \ref{ass:macroscale}. Next, we construct a local approximator using the intermediate multiscale bases $\{\Phi_{K,i}^{\operatorname{ms},m}\}_{i=1}^{d_K}$ in each oversampled subdomain $K^m$ and establish their approximation properties to $u^{\epsilon}|_{K^m}$ in Lemma \ref{lem:local-approximate}. The discrepancy of their macroscopic components is derived in Lemma \ref{lem:local-approximate-macroscopic}. Finally, we derive the approximation properties of $V_H^{\operatorname{ms},m;v}$ constructively in Theorems \ref{theorem:global-approximation} and \ref{theorem:global-approximation-type2}, respectively.

 \subsection{Local $L^2$-error estimate for the macroscopic problem}
Recall that $K^m$ is an oversampling element for $K\in\mathcal{T}_H$ and $m>0$. For any $g\in L^2(\partial K^m)$, we define the following homogeneous macroscopic problem: find $v\in L^2(K^m)$ such that
\begin{equation}\label{eq:pde-very}
\left\{\begin{aligned}
\overline{\mathcal{L}} v=-\nabla\cdot(\overline{\kappa}\nabla v)&=0 && \text{ in } K^m,\\
v&=g &&\text{ on }\partial K^m.
\end{aligned}\right.
\end{equation}
Our goal is to derive an $L^2$-estimate of the solution $v$ explicitly in terms of $\op{diam}(K^m)$.

First, we define a Lions-type variational formulation for problem \eqref{eq:pde-very} \cite[Section 6, Chapter 2]{MR0350177}. Let the test space $X(K^m)\subset H^1_{0}(K^m)$ be defined by
\begin{align}\label{eq:test-space}
X(K^m):=\left\{ z\in H^1_{0}(K^m):\overline{\mathcal{L}} z\in L^2(K^m)\right\}.
\end{align}
The test space $X(K^m)$ is endowed with the graph norm $\|\cdot\|_{X(K^m)}$:
\[
\forall z\in X(K^m):\|z\|_{X(K^m)}^2=\normE{z}{K^m}^2+\|\overline{\mathcal{L}}_{K^m} z\|_{L^2(K^m)}^2.
\]
Next, we employ a nonstandard variational form in the spirit
of the transposition method \cite{MR0350177}, and seek $v\in L^2(K^m)$ such that
\begin{align}\label{eq:nonstd-variational}
\int_{K^m}v\cdot\overline{\mathcal{L}} z\dx=-\int_{\partial K^m}g\overline{\kappa}\frac{\partial z}{\partial n}\mathrm{d}s, \quad\forall z\in X(K^m).
\end{align}
To ensure the well-definedness of the integral on the boundary $\partial K^m$, we need an $L^2$-estimate of $\frac{\partial z}{\partial n}$ for each $z\in X(K^m)$.
\begin{theorem}\label{thm:pw-Regularity}
For any $w\in L^2(K^m)$, let $z(w)\in X(K^m)$ satisfy
\begin{equation}\label{eq:pde-dual}
\left\{\begin{aligned}
\overline{\mathcal{L}} z(w)&=w && \text{ in } K^m,\\
z&=0 &&\text{ on }\partial K^m.
\end{aligned}\right.
\end{equation}
Then for some constant ${C}_{{\rm weak}} $ independent of $(\epsilon,h,H)$, there holds
\begin{align*}
\left\|\frac{\partial z(w)}{\partial n}\right\|_{L^2(\partial K^m)}
&\leq { C}_{{\rm weak}}\sqrt{H+2mh}\normL{w}{K^m}.
\end{align*}
\end{theorem}
\begin{proof}
In the proof, we abbreviate $z(w)$ to $z$. Note that $z\in X(K^m)$ is the unique solution to the following weak formulation
\begin{align*}
 \int_{K^m}\overline{\mathcal{L}}(z) q\;\dx=\int_{K^m}wq\dx,\quad \forall q\in H^1_{0}(K^m).
\end{align*}
By taking $q:=z$ and applying integration by parts, we arrive at
\begin{align*}
\overline{\alpha}\normL{\nabla z}{K^m}^2\leq\normE{ z}{K^m}^2=\int_{K^m}wz\dx.
\end{align*}
Since $z|_{\partial K^m}=0$, the Friedrichs inequality \cite {MR117419}, together with \eqref{eq:diam-subdomain}, leads to
\begin{align}\label{est:1}
\normL{\nabla z}{K^m}&\leq \frac{\op{diam}(K^m)}{\pi\overline{\alpha}}\normL{w}{K^m}\nonumber\\
&\leq \frac{\sqrt{d}}{\pi\overline{\alpha}}(H+2mh)\normL{w}{K^m}.
\end{align}
Note that the governing equation in \eqref{eq:pde-dual} implies
\begin{align*}
-\overline{\kappa}\Delta z-\nabla\overline{\kappa}\cdot\nabla z=w,\quad \text{in }K^m.
\end{align*}
By the estimate \eqref{est:1} and the fact that $\|\nabla\overline{\kappa}\|_{L^{\infty}(K^m)}\lesssim 1$ from Assumption \ref{ass:macroscale}, we derive
\begin{align*}
\normL{\Delta z}{K^m}\lesssim \normL{w}{K^m}.
\end{align*}
By the so-called Miranda-Talenti estimate on a convex domain \cite{MR0775683}, we can obtain the following {a priori} estimate
\begin{align}\label{est-delta}
|{z}|_{H^2(K^m)}&\leq \normL{\Delta z}{K^m}
\lesssim \normL{w}{K^m} .
\end{align}
Note the defining relation $
\frac{\partial z}{\partial n}=\nabla z\cdot n.$
Next, we invoke a quantitative trace theorem on an open  bounded domain with a Lipschitz boundary \cite[Theorem 1.5.1.10]{MR0775683} and obtain the following estimate on $\left\|\frac{\partial z}{\partial n}\right\|_{L^2(\partial K^m)}$:
\begin{align*}
\left\|\frac{\partial z}{\partial n}\right\|_{L^2(\partial K^m)}&\lesssim
 \delta^{1/2}(H+2mh)^{1/2}|{z}|_{H^2(K^m)}+
\delta^{-1/2}(H+2mh)^{-1/2}\|\nabla z\|_{L^{2}(K^m)},\quad \forall\delta>0,
\end{align*}
with the hidden constant being independent of the size of the subdomain $K^m$. By noting the scaling assumption \eqref{eq:c-diameter}, then taking $\delta:=1$ and combining the result with the estimates \eqref{est:1} and \eqref{est-delta}, we derive the desired assertion.
\end{proof}


Finally, we can present the main result of this part.
\begin{theorem}\label{thm:veryweak-homogenized}
Let Assumption \ref{ass:macroscale} hold, and let $g\in L^2(\partial K^m)$. Then the nonstandard variational form \eqref{eq:nonstd-variational} is well-posed. Moreover, there exists a constant ${C}_{{\rm weak}}$ independent of $(\epsilon,h,H)$ such that
\[
\normL{v}{K^m}\leq { C}_{{\rm weak}}\sqrt{H+2mh}\|g\|_{L^2(\partial K^m)}.
\]
\end{theorem}
\begin{proof}
For all $w\in L^2(K^m)$, we obtain from \eqref{eq:nonstd-variational} and Theorem \ref{thm:pw-Regularity} that
\begin{align*}
\int_{K^m} v w\;\dx\leq {C}_{\text{weak}}\sqrt{H+2mh}\normL{w}{K^m}\|g\|_{L^2(\partial K^m)}.
\end{align*}
We get the desired assertion by taking $w:=v$ in the equality.
\end{proof}

 \subsection{Local interpolation error estimate}
Now we establish the local approximation properties of the local multiscale space $\{\Phi_{K,j}^{\operatorname{ms},m}\}_{j=1}^{d_K}$ to the exact solution $u^{\epsilon}$ in each subdomain $K^m$, in the $L^2$-norm. This is proved by using a series of intermediate quantities bridging the exact solution $u^{\epsilon}$ with the candidate function $\widehat{u}_{K^m}$, which are defined by the linear interpolation over $\gamma_K^m$ and the harmonic extension of the macroscopic problem with the involved discrepancy rigorously justified by either Theorem \ref{thm:veryweak-homogenized} or Assumption \ref{ass:macroscale}. The local energy estimate in each coarse element $K$ is then proved by employing a Cacciopoli type inequality (or reverse Poincar\'{e} inequality).

First we define a crucial local problem. Let $\widehat{u}_{K^m}\in H^1_{\Gamma,0}(K^m)$ be the harmonic extension:
\begin{equation}\label{eq:111}
\left\{
\begin{aligned}
\mathcal{L} \widehat{u}_{K^m}&=0 \quad&&\text{ in } K^m\\
\widehat{u}_{K^m}&=\mathcal{I}_{\gamma_K^m}u_0\quad&&\text{ on }\gamma_K^m\\
\widehat{u}_{K^m}&=0\quad &&\text{ on }\Gamma_K^m.
\end{aligned}
\right.
\end{equation}
The boundary condition of $\widehat{u}_{K^m}$ is the linear interpolation $\mathcal{I}_{\gamma_K^m}u_0$ of the macroscopic solution $u_0$:
\begin{align*}
\mathcal{I}_{\gamma_K^m}u_0:=\sum_{j=1}^{d_K}u_0(x_{K^m,j})\psi_{K^m,j}.
\end{align*}
Note that the H\"{o}lder estimate for elliptic problems \cite[Theorems 8.22 and 8.29]{MR1814364} ensures the macroscopic solution $u_0\in C^{0,\alpha}(D)\cap H^2(D)$ for some $\alpha\in (0,1)$. Thus, the boundary condition in problem \eqref{eq:111} is well-defined.
Then the definitions of the intermediate bases \eqref{eq:lp-os} imply that $\widehat{u}_{K^m}$ can be expressed as
\begin{align}\label{eq:loc-def}
\widehat{u}_{K^m}:=\sum_{j=1}^{d_K}u_{0}(x_{K^m,j})\Phi_{K,j}^{\operatorname{ms},m}.
\end{align}
Note that the local approximator $\widehat{u}_{K^m}$ encodes the microscopic information via the heterogeneous elliptic operator $\mathcal{L}$ and the $d-1$-linear interpolation of the macroscopic solution $u_0$ over $\gamma_K^m$.
\begin{lemma}[Local interpolation error of $u^{\epsilon}-\widehat{u}_{K^m}$]\label{lem:local-approximate}
Let Assumption \ref{ass:macroscale} hold. For each $K\in\mathcal{T}_H$, let $\widehat{u}_{K^m}$ be the local nodal interpolation of $u_0|_{K^m}$ over $K^m$ using the intermediate bases defined in \eqref{eq:loc-def}.
Then the following two estimates hold
\begin{align*}
\|u^{\epsilon}-\widehat{u}_{K^m}\|_{L^2(K^m)}&\lesssim
\normL{u^{\epsilon}-u_0}{K^m}+(H+2mh)^2\left(\|u_0\|_{H^2(K^m)}+\normL{f}{K^m}\right)\\&+
\epsilon^{\delta}\normL{f}{K^m},\\
\|u^{\epsilon}-\widehat{u}_{K^m}\|_{H^1_{\kappa}(K)}&\lesssim
(mh)^{-1}\normL{u^{\epsilon}-u_0}{K^m}+(H+2mh)^{2}(mh)^{-1}\|u_0\|_{H^2(K^m)}
\\&+\left(H+2mh+(mh)^{-1}\epsilon^{\delta}\right)\normL{f}{K^m}.
\end{align*}
\end{lemma}
\begin{proof}
First, we define a local function $u_{L}\in H^1_{\Gamma,0}(K^m)$ by
\begin{equation*}
\left\{
\begin{aligned}
\mathcal{L} u_{L}&=f \quad&&\text{ in } K^m\\
u_{L}&=\mathcal{I}_{\gamma_K^m}u_0\quad&&\text{ on }\gamma_K^m\\
u_{L}&=0\quad &&\text{ on }\Gamma_K^m.
\end{aligned}
\right.
\end{equation*}
Note that this function $u_{L}$ encodes the microscale information contained in the heterogeneous elliptic operator $\mathcal{L}$, the source term $f$ and the $d-1$-linear interpolation of the macroscopic solution $u_0$ over $\gamma_K^{m}$. Let $e_{K^m}:=u^{\epsilon}-u_{L}$. We divide the rest of the proof into three steps.

\noindent \underline{Step 1.} We bound the error $e_{K^m}$ in the $L^2(K^m)$-norm.
Note that the macroscopic solution $u_0$ from \eqref{eq:model-hom-abstract} satisfies the following equation
\begin{equation*}
\left\{
\begin{aligned}
\overline{\mathcal{L}} u_0&=f \quad&&\text{ in } K^m\\
u_0&=u_0\quad&&\text{ on }\gamma_K^m\\
u_0&=0\quad&&\text{ on }\Gamma_K^m.
\end{aligned}\right.
\end{equation*}
Let $\overline{u}_{L}\in H_{\Gamma,0}^1(K^m)$ represent an approximation to $u_0|_{K^m}$ using the $d-1$-linear interpolation over $\gamma_K^m$, which satisfies
\begin{equation*}
\left\{
\begin{aligned}
\overline{\mathcal{L}} \overline{u}_{L}&=f \quad&&\text{ in } K^m\\
\overline{u}_{L}&=\mathcal{I}_{\gamma_K^m}u_0\quad&&\text{ on }\gamma_K^m\\
\overline{u}_{L}&=0\quad &&\text{ on } \Gamma_K^m.
\end{aligned}
\right.
\end{equation*}
Definition \ref{def:macro-comp} implies that  $\overline{u}_{L}$ is the macroscopic component of $u_{L}$ over $K^m$.
Then by the approximation property of linear interpolation, the following $L^2$-error estimate over $\gamma_K^m$ holds
\begin{align*}
\normL{u_0-\overline{u}_{L}}{\gamma_K^m}\lesssim (H+2mh)^{3/2}|u_0|_{H^{3/2}(\gamma_K^m)}.
\end{align*}
Upon invoking Theorem \ref{thm:veryweak-homogenized} and the trace inequality, we obtain
\begin{align}\label{eq:1/2}
\normL{u_0-\overline{u}_{L}}{K^m}&\lesssim (H+2mh)^{2}\|u_0\|_{H^{2}(K^m)}.
\end{align}
Since $\overline{u}_{L}$ is the macroscopic component of $u_L$ in the subdomain $K^m$, and $\overline{u}_{L}=u_L$ on $\partial K^m$,  Assumption \ref{ass:macroscale} indicates
\begin{align}\label{eq:000}
\normL{u_{L}-\overline{u}_{L}}{K^m}&\lesssim\epsilon^{\delta}\normL{f}{K^m}.
\end{align}
Together with \eqref{eq:1/2} and the triangle inequality, this estimate leads to
\begin{align}
\normL{e_{K^m}}{K^m}
&\lesssim
\normL{u^{\epsilon}-u_0}{K^m}+\normL{u_0-\overline{u}_{L}}{K^m}+\normL{\overline{u}_{L}-u_{L}}{K^m}\nonumber\\
&\lesssim \normL{u^{\epsilon}-u_0}{K^m}+(H+2mh)^{2}\|u_0\|_{H^2(K^m)}
+\epsilon^{\delta}\normL{f}{K^m}.\label{eq:222l2}
\end{align}
\noindent \underline{Step 2.} We estimate the error $e_{K^m}$ in the $H^1(K)$-seminorm. By definition, the error $e_{K^m}$ satisfies
\begin{equation*}
\left\{
\begin{aligned}
\mathcal{L} e_{K^m}&=0 \quad&&\text{ in } K^m\\
e_{K^m}&=u^{\epsilon}-\mathcal{I}_{\gamma_K^m}u_0\quad&&\text{ on }\gamma_K^m\\
e_{K^m}&=0\quad &&\text{ on }\Gamma_K^m.
\end{aligned}
\right.
\end{equation*}
Next, testing the first equation by $\eta_K^2e_{K^m}$, together with integration by parts and Young's inequality, we derive
\begin{align*}
\int_{K^m}\eta_K^2\kappa^{\epsilon}|\nabla {e}_{K^m}|^2\dx\leq
4\int_{K^m}\kappa^{\epsilon}
|\nabla\eta_K|^2 {e}_{K^m}^2\,\dx, 
\end{align*}
which, combining with the property of the cutoff function $\eta_K$ in \eqref{eq:pum-property}, implies
\begin{align*}
\normE{{e}_{K^m}}{K}^2&\leq\int_{K^m}\kappa^{\epsilon}\eta_K^2|\nabla {e}_{K^m}|^2\dx\\
&\leq
4\int_{K^m}\kappa^{\epsilon}
|\nabla\eta_K|^2 {e}_{K^m}^2\,\dx \\
&\leq 4(mh)^{-2}\int_{K^m}\kappa^{\epsilon} {e}_{K^m}^2\,\dx\\
&\leq 4(mh)^{-2}\beta\normL{{e}_{K^m}}{K^m}^2.
\end{align*}
Hence, together with \eqref{eq:222l2}, we obtain
\begin{align}\label{eq:1/4}
\normE{e_{K^m}}{K}
\lesssim &(mh)^{-1}\normL{u^{\epsilon}-u_0}{K^m}+(H+2mh)^{2}(mh)^{-1}\|u_0\|_{H^2(K^m)}\nonumber
\\&+(mh)^{-1}\epsilon^{\delta}\normL{f}{K^m}.
\end{align}
\noindent \underline{Step 3.} We derive the final estimate. Note that the local solution $u_{L}$ can be split into two parts $u_{L}:=u^{\op{B}}_L+\widehat{u}_{K^m}$, with the first term being the local bubble part defined by
\begin{equation}\label{eq:100011}
\left\{
\begin{aligned}
\mathcal{L} u^{\op{B}}_L&=f \quad&&\text{ in } K^m\\
u^{\op{B}}_L&=0\quad &&\text{ on }\partial K^m.
\end{aligned}
\right.
\end{equation}
Then \cite[Lemma 5.1]{fu2019edge}
implies
\begin{align*}
\normE{u^{\op{B}}_L}{K^m}\lesssim (H+2mh)\|f\|_{K^m}.
\end{align*}
Moreover, Friedrichs inequality yields
\begin{align*}
\normL{u^{\op{B}}_L}{K^m}\lesssim (H+2mh)^2\|f\|_{K^m}.
\end{align*}
These two estimates, \eqref{eq:222l2}, \eqref{eq:1/4} and the triangle inequality lead to the desired assertion.
\end{proof}

A similar argument leads to the following result. 
\begin{lemma}[Local interpolation error of $u_0-\overline{\widehat{u}}_{K^m}$]\label{lem:local-approximate-macroscopic}
For each $K\in\mathcal{T}_H$, let $\widehat{u}_{K^m}$ and $\overline{\widehat{u}}_{K^m}$ be the local nodal interpolation of $u_0|_{K^m}$ over $K^m$ using the intermediate bases defined in \eqref{eq:loc-def}, and its macroscopic component in $K^m$, respectively. Then the following two estimates hold
\begin{align*}
\|u_0-\overline{\widehat{u}}_{K^m}\|_{L^2(K^m)}\lesssim&
(H+2mh)^2\left(\|u_0\|_{H^2(K^m)}+
\normL{f}{K^m}\right),\\
\|u_0-\overline{\widehat{u}}_{K^m}\|_{H^1_{\kappa}(K)}\lesssim&
(H+2mh)^{2}(mh)^{-1}\|u_0\|_{H^2(K^m)}
+\left(H+2mh\right)\normL{f}{K^m}.
\end{align*}
\end{lemma}
\begin{proof}
By definition \ref{def:macro-comp},  $\overline{\widehat{u}}_{K^m}\in H^1_{\Gamma,0}(K^m)$  satisfies
\begin{equation*}
\left\{
\begin{aligned}
\overline{\mathcal{L}} \overline{\widehat{u}}_{K^m}&=0 \quad&&\text{ in } K^m\\
\overline{\widehat{u}}_{K^m}&=\mathcal{I}_{\gamma_K^m}u_0\quad&&\text{ on }\gamma_K^m\\
\overline{\widehat{u}}_{K^m}&=0\quad &&\text{ on }\Gamma_K^m.
\end{aligned}
\right.
\end{equation*}
The error estimate of $\normL{u_0-\overline{u}_{L}}{K^m}$ has been established in \eqref{eq:1/2}, which is harmonic in the oversampled subdomain $K^m$. Hence, the Cacciopoli type inequality as in \underline{Step 2} of the proof of Lemma \ref{lem:local-approximate} leads to
\begin{align*}
\|u_0-\overline{u}_{L}\|_{H^1_{\kappa}(K)}
\lesssim (mh)^{-1}\normL{u_0-\overline{u}_{L}}{K^m}.
\end{align*}
Note that the local solution $\overline{u}_{L}$ can be split into two parts $\overline{u}_{L}:=\overline{u}_{L}^{\op{B}}+\overline{\widehat{u}}_{K^m}$, with $\overline{u}_{L}^{\op{B}} $ satisfying \begin{equation}\label{eq:100011000}
\left\{
\begin{aligned}
\overline{\mathcal{L}} \overline{u}_{L}^{\op{B}}&=f \quad&&\text{ in } K^m\\
\overline{u}_{L}^{\op{B}}&=0\quad &&\text{ on }\partial K^m.
\end{aligned}
\right.
\end{equation}
Then an argument similar to \underline{Step 3} of the proof of Lemma \ref{lem:local-approximate} implies
\begin{align*}
\normL{\overline{u}_{L}^{\op{B}}}{K^m}+(H+2mh)\normE{\overline{u}_{L}^{\op{B}}}{K^m}\lesssim (H+2mh)^2\|f\|_{K^m}.
\end{align*}
Hence, by the triangle inequality, we obtain the desired results.
\end{proof}

Next, we define a global function $\widehat{u}\in \oplus_{K\in\mathcal{T}_H}H^1(K)$ elementwise by
\begin{align}\label{eq:666666}
\widehat{u}|_{K}:=\widehat{u}_{K^m}|_K,\quad \forall K\in\mathcal{T}_H,
\end{align}
with $\widehat{u}_{K^m}\in \op{span}\{\Phi_{K,j}^{\operatorname{ms},m}\}_{j=1}^{d_K}$, cf. \eqref{eq:loc-def}. By the definition, $\widehat{u}$ may have discontinuity along the interior edges $\mathcal{E}_H$. Nonetheless, the discontinuity can be precisely controlled. This result plays a crucial role in deriving the global error estimate in Theorem \ref{theorem:global-approximation} below.
\begin{lemma}\label{lem:discontinuity}
For any $K,K'\in\mathcal{T}_H$ such that ${K}\cap {K'}=\gamma\in\mathcal{E}_H$, there holds
\begin{align*}
\left\|\widehat{u}|_{K}-\widehat{u}|_{K'}\right\|_{C(\gamma)}
\lesssim &(mh)^{-d/2}\left(\normL{u^{\epsilon}-u_0}{K^m\cup K'^m}+(H+2mh)^2\|u_0\|_{H^2(K^m\cup K'^m)}\right)\\&+(mh)^{-d/2}\left((H+2mh)^{2}+\epsilon^{\delta} \right)\normL{f}{K^m\cup K'^m}.\nonumber
\end{align*}
\end{lemma}
\begin{proof}
Let $\widehat{e}_{K^m}:=u^{\epsilon}|_{K^m}-\widehat{u}_{K^m}$ denote the local error. By definition, it satisfies
\begin{equation*}
\left\{
\begin{aligned}
\mathcal{L} \widehat{e}_{K^m}&=f \quad&&\text{ in } K^m\\
\widehat{e}_{K^m}&=u^{\epsilon}-\mathcal{I}_{\gamma_K^m}u_0\quad&&\text{ on }\gamma_K^m\\
\widehat{e}_{K^m}&=0\quad &&\text{ on }\Gamma_K^m.
\end{aligned}
\right.
\end{equation*}
Note that the interior H\"{o}lder estimate for elliptic problems \cite[Theorems 4.1 and 4.14]{MR2777537} and H\"{o}lder's inequality imply
\begin{align}\label{eq:interior-uniform}
\|\widehat{e}_{K^m}\|_{C(\overline{K})}
\lesssim {(mh)}^{-d/2}\normL{u^{\epsilon}-\widehat{u}_{K^m}}{K^m}+(H+2mh)^{2-d/2}\normL{f}{K^m}.
\end{align}
Since $u^{\epsilon}$ is single-valued over $\gamma$, together with \eqref{eq:interior-uniform}, we obtain
\begin{align*}
\sup_{x\in\gamma}\left|\widehat{u}|_{K}(x)-\widehat{u}|_{K'}(x)\right|&\leq
\sup_{x\in\gamma}\left(\left|\widehat{u}_{K^m}(x)-u^{\epsilon}(x)\right|+\left|u^{\epsilon}(x)-\widehat{u}_{K'^m}(x)\right|\right)\\
&\lesssim {(mh)}^{-d/2}\left(\normL{\widehat{e}_{K^m}}{K^m}+\normL{\widehat{e}_{{K'}^m}}{{K'}^m}\right)\\
&+(H+2mh)^{2-d/2}\normL{f}{K^m\cup K'^m}.
\end{align*}
Finally, Lemma \ref{lem:local-approximate} yields the desired assertion.
\end{proof}

Similarly, we can estimate the jump of the macroscopic component of $\widehat{u}$ over the interior edges $\mathcal{E}_H$. 
\begin{lemma}\label{lem:discontinuity-macroscopic}
For any $K,K'\in\mathcal{T}_H$ such that ${K}\cap {K'}=\gamma\in\mathcal{E}_H$, there holds
\begin{align*}
\left\|\overline{\widehat{u}}_{K^m}-\overline{\widehat{u}}_{K'^m}\right\|_{C(\gamma)}
\lesssim &(mh)^{-d/2}(H+2mh)^2\left(\|u_0\|_{H^2(K^m\cup K'^m)}+\normL{f}{K^m\cup K'^m}\right).
\end{align*}
\end{lemma}
\subsection{Global interpolation error estimates}
The next two theorems give the global interpolation error estimates, which are the main results of this work.
\begin{theorem}[Global approximation property of $V_H^{\operatorname{ms},m;\phi}$]\label{theorem:global-approximation}
Let Assumption \ref{ass:macroscale} hold. Let $u^{\epsilon}$ be the solution to problem \eqref{eq:model} and $V_H^{\operatorname{ms},m;\phi} $ be the multiscale finite element space with $m$ fine-scale oversampled layers defined in \eqref{eq:v-ov}. Then there holds
\begin{align*}
\min_{v_H^{\operatorname{ms},m}\in V_H^{\operatorname{ms},m;\phi}}\normb{u^{\epsilon}-v_H^{\operatorname{ms},m}}
&\lesssim
\left((mh)^{-1}+\left(\frac{H+2mh}{mh}\right)^{d/2}(H+2mh)^{-1}\right)\epsilon^{\delta}\normL{f}{D}\\
&+\left(\frac{H+2mh}{mh}\right)^{d/2}(H+2mh)\left(\normL{f}{D}+\|u_0\|_{H^2(D)}\right)\\
&+(mh)^{-1}(H+2mh)^{2}\|u_0\|_{H^2(D)}.
\end{align*}
\end{theorem}
\begin{proof}
It suffices to construct a specific function $w_H^{\operatorname{ms},m}\in V_H^{\operatorname{ms},m;\phi}$ of the form
\begin{align*}
w_H^{\op{ms},m}:=\sum_{i\in\mathcal{J}_H} \widehat{c}_i\phi_{i}^{\op{ms},m},
\end{align*}
such that
\begin{align*}
\normb{u^{\epsilon}-w_H^{\operatorname{ms},m}}&\lesssim
\left((mh)^{-1}+\left(\frac{H+2mh}{mh}\right)^{d/2}(H+2mh)^{-1}\right)\epsilon^{\delta}\normL{f}{D}\\
&+\left(\frac{H+2mh}{mh}\right)^{d/2}(H+2mh)\left(\normL{f}{D}+\|u_0\|_{H^2(D)}\right)\\
&+(mh)^{-1}(H+2mh)^{2}\|u_0\|_{H^2(D)}.
\end{align*}
Our aim is to identify the sequence $\{\widehat{c}_i\}_{i\in\mathcal{J}_H}$ using the local properties of the exact solution $u^{\epsilon}$. Since $\{\phi_{i}^{\op{ms},m}\}_{i\in\mathcal{J}_H}$ has the Lagrange interpolation property over the sequence of interior nodes $\{x_i\}_{i\in\mathcal{J}_H}$ so that $\widehat{c}_i=w_H^{\operatorname{ms},m}(x_i)$ for all $i\in\mathcal{J}_H$.
We divide the proof into four steps.

\noindent \underline{Step 1.} First, we derive a local representation of $\widehat{u}|_K$. Let $\widehat{u}\in \oplus_{K\in\mathcal{T}_H}H^1(K)$ be defined in \eqref{eq:666666}. Then by Lemma \ref{lem:local-approximate}, the global function $\widehat{u}$ has good local approximation properties to the restriction $u^{\epsilon}|_{K}$ of the exact solution $u^{\epsilon}$ on the subdomain $K$. Moreover, by definition \eqref{eq:loc-def},  the equivalence of the intermediate bases and the actual bases in the sense
\begin{align*}
\op{span}\left\{\Phi_{K,j}^{\operatorname{ms},m}:j=1,\cdots,d_K\right\}=\op{span}\left\{\phi_{K,j}^{\operatorname{ms},m}:j=1,\cdots,d_K\right\},
\end{align*}
and the Lagrange interpolation property of $\{\phi_{i}^{\op{ms},m}\}_{i\in\mathcal{J}_H}$ over the sequence of interior nodes $\{x_i\}_{i\in\mathcal{J}_H}$, there holds
\begin{align*}
\widehat{u}|_{K}:=\sum_{j=1}^{d_K}u_{0}(x_{K^m,j})\Phi_{K,j}^{\operatorname{ms},m}
:=\sum_{j=1}^{d_K}\widehat{u}_{K^m}(x_{K,j})\phi_{K,j}^{\operatorname{ms},m}.
\end{align*}

\noindent \underline{Step 2.} Next we define a candidate function $w_H^{\op{ms},m}\in V_H^{\operatorname{ms},m;\phi}$ by averaging the nodal values of $\widehat{u}$. Note that the constructed auxiliary global function $\widehat u$ defined in \eqref{eq:666666} fails to fall into the multiscale space, i.e., $\widehat{u}\notin V_H^{\operatorname{ms},m;\phi}$, due to nonunique nodal values at $x_i$ for $i\in\mathcal{J}_H$. We present a candidate function to approximate $u^{\epsilon}|_K$ on each coarse element $K$ by averaging the nodal values:
\begin{align}
w_H^{\op{ms},m}:=\sum_{i\in\mathcal{J}_H} \widehat{c}_i\phi_{i}^{\op{ms},m}\in V_H^{\operatorname{ms},m;\phi},\quad \mbox{with }
\label{eq:loc-glo}
\widehat{c}_i:=\frac{1}{d_i}\sum_{K\in\mathcal{T}_H:\;x_i\in {K}}\widehat{u}|_{K}(x_{i }),
\end{align}
with $d_i$ the number of coarse elements sharing the nodes $x_{i}$ for $i\in\mathcal{J}_H$.

\noindent \underline{Step 3.} We estimate the local error between $\widehat{u}$ and $w_H^{\op{ms},m}$.
 For any $K,K'$ such that ${K}\cap {K'}=\gamma\in\mathcal{E}_H$, Lemma \ref{lem:discontinuity} implies that for all $i\in\mathcal{J}_H$ with $x_i\in {\gamma}$, there holds
\begin{align}\label{eq:555}
\left|\widehat{u}|_{K}(x_{i})-\widehat{u}|_{K'}(x_{i})\right|\lesssim&(mh)^{-d/2}\left(\normL{u^{\epsilon}-u_0}{K^m\cup K'^m}+(H+2mh)^2\|u_0\|_{H^2(K^m\cup K'^m)}\right)
\nonumber
\\&+(mh)^{-d/2}\left((H+2mh)^{2}+\epsilon^{\delta} \right)\normL{f}{K^m\cup K'^m}.
\end{align}
Next, we restrict to one coarse element $K\in \mathcal{T}_H$. Let $\{\ell^K_j\}_{j=1}^{d_K}\subset \mathcal{J}_H$ be the global nodal index corresponding to the local nodes $\{x_{K,j}\}_{j=1}^{d_K}$, i.e., $x_{K,j}=x_{\ell_j^K}$ for $j=1,\cdots,d_K$. We derive from the estimates \eqref{eq:loc-glo} and \eqref{eq:555} that
\begin{align}\label{eq:77777777}
\left|\widehat{c}_{\ell_j^K}-\widehat{u}|_{K}(x_{\ell_j^K})\right|
&=\bigg|\frac{1}{d_{\ell_j^K}}\sum_{K'\in\mathcal{T}_H:\;x_{\ell_j^K}\in {K'}}(\widehat{u}|_{K'}(x_{\ell_j^K})-\widehat{u}|_{K}(x_{\ell^K_j}))\bigg|.
\end{align}
Note that the index $\{\ell_j^K\}_{j=1}^{d_K}$ gives the local representation of $w_H^{\op{ms},m}$ in each coarse element $K$:
\begin{align*}
w_H^{\op{ms},m}|_K=\sum_{j=1}^{d_K}\widehat{c}_{\ell_j^K}\phi_{K,j}^{\op{ms},m}.
\end{align*}
Consequently, we obtain
\begin{align}
\normE{w_H^{\op{ms},m}-\widehat{u}_{K^m}}{K}&=\normE{\sum_{j=1}^{d_K}\left(\widehat{c}_{\ell_j^K}-\widehat{u}|_{K}(x_{\ell_{j}^K})\right)\phi_{K,j}^{\operatorname{ms},m}}{K}\nonumber\\
&\leq \sum_{j=1}^{d_K}\left|\widehat{c}_{\ell_j^K}-\widehat{u}|_{K}(x_{\ell_j^K})\right|\normE{\phi_{K,j}^{\operatorname{ms},m}}{K},\label{eq:local-error-coeff0}
\end{align}
By inserting \eqref{eq:77777777} and \eqref{eq:555} into the inequality \eqref{eq:local-error-coeff0}, together with \eqref{eq:loc-harmonic-energy}, we derive
\begin{equation}\label{eq:local-error-coeff}
\begin{aligned}
\normE{w_H^{\op{ms},m}-\widehat{u}|_{K}}{K}\lesssim&
\left(\frac{H+2mh}{mh}\right)^{d/2}(H+2mh)^{-1}\normL{u^{\epsilon}-u_0}{N(K)}\\
&+\left(\frac{H+2mh}{mh}\right)^{d/2}
{(H+2mh)}
\|u_0\|_{H^2(N(K))}
\\
&+\left(\frac{H+2mh}{mh}\right)^{d/2}
 (H+2mh)^{-1}\left(\epsilon^{\delta}+(H+2mh)^2\right)
\normL{f}{N(K)},
\end{aligned}
\end{equation}
where $N(K):=\bigcup\{T^m: T\in\mathcal{T}_H,\; {K}\cap {T}\neq\emptyset\}$ denotes the union of $K^m$ and adjacent oversampled subdomains for $K\in\mathcal{T}_H$.

\noindent \underline{Step 4.} At this step, we derive a global error estimate. By the triangle inequality,
\begin{align*}
\normb{u^{\epsilon}-w_H^{\operatorname{ms},m}}^2&= \sum_{K\in\mathcal{T}_H}\normE{u^{\epsilon}-w_H^{\operatorname{ms},m}}{K}^2\\
&\leq 2\sum_{K\in\mathcal{T}_H}\left(\normE{u^{\epsilon}-\widehat{u}_{K^m}}{K}^2+\normE{\widehat{u}|_{K}-w_H^{\operatorname{ms},m}}{K}^2\right).
\end{align*}
Using Lemma \ref{lem:local-approximate} and \eqref{eq:local-error-coeff}, we obtain
\begin{align*}
\normb{u^{\epsilon}-w_H^{\operatorname{ms},m}}^2\lesssim& \sum_{K\in\mathcal{T}_H}\left((mh)^{-1}+(H+2mh)^{-1}\left(\frac{H+2mh}{mh}\right)^{d/2}\right)^2\normL{u^{\epsilon}-u_0}{N(K)}^2\\
&+\left((H+2mh)^{2}(mh)^{-1}
+\left(\frac{H+2mh}{mh}\right)^{d/2}(H+2mh)
\right)^2\|u_0\|_{H^2(N(K))}^2\\
&+\left(\frac{\epsilon^{\delta}}{mh}+\left(\frac{H+2mh}{mh}\right)^{d/2}(H+2mh)^{-1}\left(\epsilon^{\delta}+(H+2mh)^2\right)\right)^2\normL{f}{N(K)}^2.
\end{align*}
Combining this estimate with the finite overlapping condition \eqref{eq:finite-overlap} and Assumption \ref{ass:macroscale} leads to
\begin{align*}
\normb{u^{\epsilon}-w_H^{\operatorname{ms},m}}\lesssim & \left(\frac{\epsilon^{\delta}}{mh}+\left(\frac{H+2mh}{mh}\right)^{d/2}(H+2mh)^{-1}\left(\epsilon^{\delta}+(H+2mh)^2\right)\right)\normL{f}{D}\\
&+\left((H+2mh)^{2}(mh)^{-1}
+\left(\frac{H+2mh}{mh}\right)^{d/2}(H+2mh)
\right)\|u_0\|_{H^2(D)}.
\end{align*}
This completes the proof of the theorem.
\end{proof}
\begin{theorem}[Global approximation property of $V_H^{\operatorname{ms},m;\varphi}$]\label{theorem:global-approximation-type2}
Let Assumption \ref{ass:macroscale} hold. Let $u^{\epsilon}$ be the solution to problem \eqref{eq:model} and $V_H^{\operatorname{ms},m;\varphi} $ be the multiscale finite element space with $m$ fine-scale oversampled layers defined in \eqref{eq:v-ov}. Then there holds
\begin{align*}
\min_{v_H^{\operatorname{ms},m}\in V_H^{\operatorname{ms},m;\varphi}}\normb{u^{\epsilon}-v_H^{\operatorname{ms},m}}
\lesssim
&(mh)^{-1}\epsilon^{\delta}\normL{f}{D}+(mh)^{-1}(H+2mh)^{2}\|u_0\|_{H^2(D)}\\
&+\left(\frac{H+2mh}{mh}\right)^{d/2}(H+2mh)\left(\normL{f}{D}+\|u_0\|_{H^2(D)}\right).
\end{align*}
\end{theorem}
\begin{proof}
We present a candidate function to approximate $u^{\epsilon}|_K$ on each coarse element $K$ by averaging the nodal values:
\begin{align}
w_H^{\op{ms},m}:=\sum_{i\in\mathcal{J}_H} \widehat{c}_i\varphi_{i}^{\op{ms},m}\in V_H^{\operatorname{ms},m;\varphi},\quad \mbox{with }
\label{eq:loc-glo-type2}
\widehat{c}_i:=\frac{1}{d_i}\sum_{K\in\mathcal{T}_H:\;x_i\in\overline{K}}\overline{\widehat{u}}_{K^m}(x_{i }).
\end{align}
Note that $\widehat{u}_{K^m}$ can be expressed as 
\begin{align*}
{\widehat{u}}_{K^m}=\sum_{i=1}^{d_K} \overline{\widehat{u}}_{K^m}(x_{K,i})\varphi_{K,i}^{\op{ms},m}. 
\end{align*}
Then by repeating the argument at \underline{Step 3} of the proof of Theorem \ref{theorem:global-approximation} leads to the error estimate between ${\widehat{u}}_{K^m}$ and ${w}_H^{\op{ms},m}$:
\begin{equation*}
\begin{aligned}
\normE{{w}_H^{\op{ms},m}-{\widehat{u}}_{K^m}}{K}\lesssim
&\left(\frac{H+2mh}{mh}\right)^{d/2}
{(H+2mh)}
\left(\|u_0\|_{H^2(N(K))}+\normL{f}{N(K)}\right).
\end{aligned}
\end{equation*}
The remaining proof follows from \underline{Step 4} of the proof of Theorem \ref{theorem:global-approximation}. 
\end{proof}
\begin{remark}[Influence of oversampling size $mh$]
Theorems \ref{theorem:global-approximation} and \ref{theorem:global-approximation-type2} indicate that oversampling is important to control the interpolation error but not the bigger the better. Meanwhile, there is no convergence if $mh=\mathcal{O}(1)$ due to a lack of approximation to the local bubble functions, cf. \eqref{eq:100011}, and insufficient  information transferred to $\widehat{u}_{K^m}$, cf. the estimate \eqref{eq:1/2}.
\end{remark}
\begin{remark}
Let Assumption \ref{ass:macroscale} hold. Note the assumption $\epsilon\ll H$, cf. \eqref{eq:c-diameter}. Then
Theorems \ref{theorem:global-approximation} and  \ref{theorem:global-approximation-type2} imply that fir either Type 1 or Type 2 LMS, if $mh\approx H$, then
\begin{align*}
\min_{v_H^{\operatorname{ms},m}\in V_H^{\operatorname{ms},m;v}}\normb{u^{\epsilon}-v_H^{\operatorname{ms},m}}=\mathcal{O}\left(H+\frac{\epsilon^{\delta}}{H}\right).
\end{align*}
\end{remark}

\section{Interpolation error for special structures}\label{sec:structure}
In this section, we provide two special structures on $\kappa^{\epsilon}$ for which Assumption \ref{ass:macroscale} holds. The two settings have been extensively studied in the literature.
\subsection{Rapidly oscillating periodic coefficient}\label{sec:periodic}
We discuss the case that the permeability coefficient $\kappa^{\epsilon}$ is periodic and highly oscillatory, with $\epsilon$ denoting the period size.
\begin{assumption}[Highly oscillatory periodic coefficients]\label{ass:periodic-structure}
The permeability $\kappa^{\epsilon}(x):=\kappa(\frac{x}{\epsilon})$ with
$\kappa\in \mathcal{M}(\alpha,\beta;Y)$, $\kappa^T=\kappa$ and $Y:=[0,1]^d$ denoting a unit cell.
\end{assumption}

First, we recall known homogenization theory related to problem \eqref{eq:model} with highly oscillatory periodic coefficients.
Analogous to the standard homogenization theory, we denote $y:=x/\epsilon$ to be the fast variable, and $x$ to be the slow variable.
Let $L^{2}_{\#}(Y):=\{ u\in L^{2}(Y):  u \text{ is $Y$-periodic}\}$. Similarly, we can define $H^{1}_{\#}(Y)$. Let
$V_{\#}(Y):=\{v\in H^{1}_{\#}(Y): \int_{Y}v\;\dy=0\}$.

For all $k=1,\cdots,d$, let $\chi_k\in V_{\#}(Y)$ be the solution to the following {\it cell problem}:
\begin{equation*}
 -\nabla\cdot \left ( \kappa
\nabla\chi_k\right ) =\nabla\cdot(\kappa \bm{e}_k),\quad \text{in } Y,
\end{equation*}
where $\{\bm{e}_k\}_{k=1}^d$ is the canonical basis in $\mathbb{R}^d$.
The constant effective coefficient $\overline{\kappa}$ is defined by
\begin{equation*}
\overline{\kappa}\bm{e}_k := \int_Y \kappa(\bm{e}_k+\nabla \chi_k
)  \,\dy,\quad \text{ for } k=1,\cdots,d.
\end{equation*}
Then the homogenized solution is defined by solving $u_0\in V$ such that
\begin{equation}\label{eq:model-hom}
\left\{\begin{aligned}
\overline{\mathcal{L}}u_0:=-\nabla\cdot\left(\overline{\kappa}\nabla u_0\right)&=f &&\quad\text{ in }D,\\
u_0&=0 &&\quad\text{ on } \Gamma.
\end{aligned}\right.
\end{equation}
Moreover, the standard elliptic regularity theory implies
\begin{align}\label{eq:u-0-reg}
|u_0|_{H^2(D)}\lesssim \normL{f}{D}.
\end{align}
Together with \cite[Theorem 3.4.3]{MR3838419}, the following $L^2$-error estimate holds:
\begin{align}\label{eq:1/3}
\normL{u^{\epsilon}-u_0}{D}\lesssim \epsilon \normL{f}{D}.
\end{align}
Next, we show that Assumption \ref{ass:macroscale} holds with $\delta=1$.
\begin{lemma}\label{lem:verify-macroscopic}
Let Assumption \ref{ass:periodic-structure} be fulfilled. Then \eqref{eq:000} holds with $\delta=1$. Hence, Assumption \ref{ass:macroscale} holds.
\end{lemma}
\begin{proof}
Note that $\overline{u}_{L}$ is the homogenized solution, or the macroscopic component, to $u_L$ in the subdomain $K^m$ and that
$$u_{L}|_{\partial K^m}=\overline{u}_{L}|_{\partial K^m}\in H^1(\partial K^m).$$
By \cite[Theorem 3.4.3]{MR3838419} and the Cauchy-Schwarz inequality, we deduce
\begin{align*}
\normL{u_{L}-\overline{u}_{L}}{K^m}
&\lesssim|K^m|^{1/p_1} \norm{u_{L}-\overline{u}_{L}}_{L^p(K^m)}\\
&\lesssim |K^m|^{1/p_1} \frac{\epsilon}{H} \|u_0\|_{w^{2,q}(K^m)}\\
&\lesssim |K^m|^{1/p_1}|K^m|^{1/q_1} \frac{\epsilon}{H} \|u_0\|_{H^2(K^m)}\\
&= |K^m|^{1/p_1+1/q_1} \frac{\epsilon}{H} \|u_0\|_{H^2(K^m)},
\end{align*}
where the second inequality is due to the scale-invariant $L^p$-error estimate for the homogenized solution
with $p=\frac{2d}{d-1}$ and $q=\frac{2d}{d+1}$ denoting its dual exponent
\cite[Theorem 3.4.3]{MR3838419}, and the first and third inequalities are due to the Cauchy-Schwarz inequality
with $p_1=q_1=2d$. Hence, by combining the estimates \eqref{eq:diam-subdomain}, \eqref{eq:c-diameter} and \eqref{eq:u-0-reg}, we obtain the bound \eqref{eq:000} with $\delta=1$.
\end{proof}

\begin{remark}[Consistency with standard approach]
Let Assumption \ref{ass:periodic-structure} hold. Note the assumption $\epsilon\ll H$, cf. \eqref{eq:c-diameter}. Then
Theorem \ref{theorem:global-approximation-type2} and Lemma \ref{lem:verify-macroscopic} imply that if $mh\approx H$, then
\begin{align*}
\min_{v_H^{\operatorname{ms},m}\in V_H^{\operatorname{ms},m;\varphi}}\normb{u^{\epsilon}-v_H^{\operatorname{ms},m}}=\mathcal{O}\left(H+\frac{\epsilon}{H}\right).
\end{align*} 
Note that this result is consistent with the conforming error defined in \cite[Equation (3.8)]{MR1740386}, which is proved using the standard approach in the context of highly oscillatory periodic coefficients.
\end{remark}

\subsection{Rapidly oscillating locally periodic coefficients}
Next, we discuss the case that the permeability coefficient $\kappa^{\epsilon}$ is locally periodic and rapidly oscillating with a period size $\epsilon$ \cite{MR4578546}.
\begin{assumption}[Highly oscillatory locally periodic coefficients]\label{ass:periodic-structure-local}
The permeability coefficient $k^\epsilon$ takes the form
$\kappa^{\epsilon}(x)=\kappa(x,x/\epsilon)
\in C^{0,s}(\overline{D};L^{\infty}(Y))^{d\times d}$, which is H\"{o}lder continuous of order $s\in (0,1]$ in the first variable and periodic in the second variable. Moreover, $\kappa^{\epsilon}(x)
\in\mathcal{M}(\alpha,\beta;D)$. 
\end{assumption}
First, we define the cell problem for each $x\in D$. For all $k=1,\cdots,d$, let $N_k(x,x/\epsilon)\in V_{\#}(Y)$ be the
solution to the following {\it cell problem}
\begin{equation*}
 -\nabla\cdot \left ( \kappa(x,\cdot)
\nabla N_k(x,\cdot)\right ) =\nabla\cdot(\kappa(x,\cdot) \bm{e}_k)  \quad \text{ in } Y.
\end{equation*}
Then $N_k(x,x/\epsilon)\in C^{0,s}(\overline{D};V_{\#}(Y))^{d\times d}$ is H\"{o}lder continuous of order $s\in (0,1]$ in the first variable \cite[Lemma 4.1]{MR4578546}. Hence the effective coefficient $\overline{\kappa}\in C^{0,s}(\overline{D})^{d\times d}$ is H\"{o}lder continuous of order $s$ \cite[Section 4]{MR4578546}, which is defined by
\begin{equation*}
\overline{\kappa}(x)\bm{e}_k :=  \int_Y \kappa(x,y)(\bm{e}_k+\nabla N_k(x,y)
)  \,\dy , \quad \text{for } k=1,\cdots,d.
\end{equation*}
Note that the computation of the effective coefficient $\overline{\kappa}$ involves an infinite number of cell problems, and thus impractical. This represents a sharp contrast to the periodic setting in Section \ref{sec:periodic}, for which one cell problem is sufficient.

Then the homogenized solution $u_0$ is defined by finding $u_0\in V$ such that
\begin{equation}\label{eq:model-hom-loc-period}
\left\{\begin{aligned}
\overline{\mathcal{L}}u_0:=-\nabla\cdot\left(\overline{\kappa}(x)\nabla u_0\right)&=f &&\quad\text{ in }D,\\
u_0&=0 &&\quad\text{ on } \Gamma.
\end{aligned}\right.
\end{equation}
It is known that the following a priori error estimate holds \cite[Theorem 6.3]{MR4578546}:
\begin{align*}
\normL{u^\epsilon-u_0}{D}\lesssim \epsilon^{s}\normL{f}{D}.
\end{align*}
Moreover, \eqref{eq:000} holds with $\delta=s$ since $\overline{u}_{L}$ is the homogenized solution, or the macroscopic component, to $u_L$ in the subdomain $K^m$. Hence, Assumption \ref{ass:macroscale} holds with $\delta:=s$ for the specific setting. 

\section{Conclusion}\label{sec:conc}
In this work, we have established a new interpolation error analysis of MsFEMs with oversampling for elliptic problems with highly oscillatory rough coefficients, under the assumption of the existence of a macroscopic problem with $L^2$-accuracy. One distinct feature of the new analysis is that it is independent of the solutions to the cell problems, and the first-order corrector in the case with periodic highly oscillating coefficients. The analysis leads to an interpolation error $\mathcal{O}(H+\frac{\epsilon}{H})$ when the oversampled subdomain has a distance of one coarse element away from the target coarse element in the case of highly oscillating (locally) periodic coefficients. Theoretically, it is of great importance to investigate the stability of various global MsFEM formulations in a general setting.

\section*{Acknowledgement}
I would like to thank Daniel Peterseim (University of Augsburg, Germany) for fruitful discussion, carefully reading an early version of the manuscript and constructive comments. The project was started during the follow-up Workshop "Multiscale Problems: Algorithms, Numerical Analysis and Computation", Hausdorff Institute of Mathematics, Bonn, Germany, 2024.

\bibliographystyle{siam}
\bibliography{reference}

\end{document}